\documentclass[11pt]{article}
\usepackage [dvips]{graphics}
\usepackage[centertags]{amsmath}
\usepackage{amsfonts}
\usepackage{amssymb}
\usepackage{amsthm}
\usepackage{newlfont}
\usepackage{stmaryrd}

\theoremstyle{plain}
\newtheorem{thm}{Theorem}[section]
\newtheorem{cor}[thm]{Corollary}
\newtheorem{lem}[thm]{Lemma}
\newtheorem{prop}[thm]{Proposition}
\newtheorem{exam}[thm]{Example}
\newtheorem{rem}[thm]{Remark}

\def\sqr#1#2{{\vcenter{\vbox{\hrule height.#2pt
              \hbox{\vrule width.#2pt height#1pt \kern#1pt \vrule
width.#2pt}
              \hrule height.#2pt}}}}
\def\3n{\negthinspace \negthinspace \negthinspace }
\def\2n{\negthinspace \negthinspace }
\def\1n{\negthinspace }

\def\cA{{\cal A}}

\def\no{\noindent}

\def\ms{\medskip}
\def\bs{\bigskip}

\def\dim{\hbox{\rm dim$\,$}}

\def\span{\hbox{\rm span$\,$}}
\def\tr{\hbox{\rm tr$\,$}}

\def\({\Big (}
\def\){\Big )}
\def\[{\Big[}
\def\]{\Big]}

\def\be{\begin{equation}}
\def\bel{\begin{equation}\label}
\def\ee{\end{equation}}
\def\bea{\begin{eqnarray}}
\def\eea{\end{eqnarray}}
\def\bt{\begin{theorem}}
\def\et{\end{theorem}}
\def\bc{\begin{corollary}}
\def\ec{\end{corollary}}
\def\bl{\begin{lemma}}
\def\el{\end{lemma}}
\def\bp{\begin{proposition}}
\def\ep{\end{proposition}}
\def\br{\begin{remark}}
\def\er{\end{remark}}
\def\ba{\begin{array}}
\def\ea{\end{array}}
\def\bd{\begin{definition}}
\def\ed{\end{definition}}

\newcommand{\poly}{\mathbb{C}[z_1,\ldots,z_d]}
\newcommand{\Hol}{\mathrm{Hol~}}
\newcommand{\Id}{\mathrm{Id}}

\makeatletter
   
   \@addtoreset{equation}{section}
\makeatother

\begin{document}

\title{\bf Essential normality of homogenous quotient modules over the polydisc: distinguished variety case}
\author{ Penghui Wang\thanks{ Partially supported by NSFC(No.11101240),   E-mail: phwang@sdu.edu.cn }
\quad Chong Zhao \thanks{ Partially supported by Shandong Province Natural Science Foundation ZR2014AQ009 and The Fundamental Research Funds of Shandong University 2015GN017, {
E-mail:} { chong.zhao@sdu.edu.cn}. \ms}
 \\ \\
 Department of Mathematics, Shandong University\\
Jinan, Shandong 250100, The People's Republic of China\\
}

\maketitle
\begin{abstract}
In the present paper, we study the essential normality of quotient modules over the polydisc. It is shown that if the zero variety of homogenous ideal $I$ is a distinguished variety, then its quotient module is $(1,\infty)$-essentially normal. Moreover, we study the boundary representation of quotient modules.
\end{abstract}

\bs

\no{\bf 2000 MSC}. 47A13, 46H25

\bs

\no{\bf Key Words}. Essential normality, Hardy space over polydisc, quotient module, boundary representation
\section{Introduction}
~~~~Let $\mathbb D^d=\{(z_1,\ldots,z_d): |z_i|<1,i=1,\ldots,d\}$ be the unit polydisc in $\mathbb C^d$, and  $H^2(\mathbb D^d)$ the Hardy space over $\mathbb D^d$. It plays an important role in multivariable operator theory and function thoery\cite{Ru}. The present paper is devoted to study the essential normality of homogenous quotient modules of $H^2(\mathbb D^d)$ when $d\geq 3$, and their boundary representations.

The theory of Hilbert module, introduced by Douglas and Paulsen \cite{DP}, is a natural language in the study of multivariable operator theory. Given a tuple $\underline{T}=(T_1,\ldots,T_d)$ of commuting operators on a Hilbert space $H$, one can naturally make $H$ into a Hilbert module over the polynomial ring $\mathbb C[z_1,\ldots,z_d]$, with the module action defined as
$$
p\cdot x=p(T_1,\ldots, T_d)x, \quad p\in \mathbb C[z_1,\ldots,z_d], \quad x\in H.
$$
Motivated by the BDF-theory\cite{BDF}, one of the fundamental problems in the Hilbert module theory is to study the essential commutativity of the $C^*$-algebra $C^*(\underline{T})$ generated by $Id,T_1,\ldots,T_d$, where $Id$ is the identity on $H$. When this happens, $H$ is said to be essentially normal. In \cite{Ar1}, Arveson conjectured that graded submodules of the $d$-shift module over the unit ball are essentially normal, and much work has been done along this line, such as \cite{Ar1,Ar3,Dou1,Dou2,DW,Guo,GWk1,GWk2,KS} and references therein.

In the present paper we consider the polydisc version of Arveson's conjecture. It is easy to see that all the nontrivial submodules are not essentially normal provided $d\geq2$. Hence in the case of the unit polydisc, one can only consider essential normality of quotient modules. Let $\mathcal{M}$ be a submodule of $H^2(\mathbb D^d)$, denote by $\mathcal{N}=\mathcal{M}^\perp$ and write
$$S_{z_i}=P_\mathcal{N}M_{z_i}\mid_\mathcal{N},$$
which is the compression of the multiplication operator $M_{z_i}$ on $\mathcal{N}$. Then $\mathcal{N}$ is naturally equipped with a $\mathbb C[z_1,\ldots,z_d]$-module structure by the tuple $(S_{z_1},\ldots, S_{z_d})$, and called a quotient module of $H^2(\mathbb D^d)$.

The first result along this line was due to Douglas and Misra \cite{DM}, who showed that some quotient modules are essentially normal and some are not. By restricting the Hardy space to the diagonal, Clark \cite{Cla} identified the quotient module generated by $\{B_i(z_i)-B_j(z_j);i,j=1,\cdots,d\}$ for finite Blaschke products $B_i(z_i)$ as a kind of Bergman space on some variety, hence it is essentially normal. The essential normality of (quasi-)homogenous quotient modules for $d=2$ was completely characterized by Guo and the first author \cite{GWp1,GWp2}, and $p$-essential normality was studied in \cite{GWZ}.

Briefly, the answer to the polydisc version of Arveson's conjecture  is totally different from the original question on the unit ball.

In this paper, we mainly consider homogenous quotient modules of $H^2(\mathbb D^d)$, which is a continuation of \cite{GWp1}. Let $I$ be a homogenous ideal of $\mathbb C[z_1,\ldots, z_d]$ and $V$ be its zero variety. In the case of the unit ball, Douglas, Tang and Yu \cite{DTY} proved that, if $\dim_{\mathbb{C}}V\geq 2$ then both $[I]$ and $[I]^\perp$ are essentially normal. Similar to \cite{GWp2,Wa}, we have the following proposition.
\begin{prop}
Let $V$ be the zero variety of a homogenous ideal $I$. If $\dim_{\mathbb{C}}V\geq2$, then the quotient module $[I]^\perp$ of $H^2(\mathbb{D}^d)$ is not essentially normal.
\end{prop}

The proof of this proposition is based on the Berezin transform and analysis of zero variety, which is almost the same as that given in \cite{GWp2}, and we will not perform it here. Based on this observation, to study the essential normality of homogenous quotient modules, we will only consider the case $\dim_{\mathbb{C}}Z(I)=1$. We will see in Proposition \ref{prop:varietydim} that, if the zero variety is distinguished, then $\dim_{\mathbb{C}}Z(I)=1$. Therefore in the present paper, we mainly consider the distinguished variety case. The following theorem is the main result of this paper.

\begin{thm}
    If $I\subset\poly$ is a homogeneous ideal such that $Z(I)\cap\partial\mathbb{D}^d\subset\mathbb{T}^d$, then the quotient module $\mathcal{M}=[I]^\bot$ is $(1,\infty)$-essentially normal.
\end{thm}
\begin{rem}
The distinguished variety in the bidisk was introduced in \cite{AM}. And the previous work \cite{GWp1,GWp2,Wa} suggests that the essential normality of quotient modules of the Hardy module on the polydisc is closely related to the distinguished variety.
\end{rem}

The key-step to prove the theorem is to consider the $p$-essential normality of $[J^n]^\perp$ for some prime ideal $J$ such that $Z(J)$ is a distinguished variety and $(1,1,\ldots,1)\in Z(J)$. To this end, the main tool is the restriction map $r:H^2(\mathbb D^d)\to Hol(\mathbb D)$ defined by $r(f)(z)=f(z,\ldots z)$, which was introduced in \cite{FR1} for the bidisk case.

\begin{thm}Let $\mathcal N=[J^N]^\perp$, then  for $f_1,f_2\in\poly$,
    $$\tr[S_{f_1}^*,S_{f_2}]=\binom{d+N-2}{d-1}\langle (rf_2)^\prime,(rf_1)^\prime\rangle_{L_a^2(\mathbb D)}.$$
\end{thm}

For the bidisk case, such a kind of trace formula was given in \cite{GWZ}. However, it will be seen that, the high dimensional case is far more complicated.

The theory of boundary representation of $C^*$-algebra developed by Arveson \cite{Arv,Arv1} plays an important role in multi-variable operator theory. Recently, Kennedy and Shalit \cite{KS} showed that for the $d$-shift module on the unit ball, the essential normality is closely connected to the boundary representation. In the last section of the present paper, we also consider the boundary representation for the distinguished homogenous quotient modules of $H^2(\mathbb D^d)$ which is related to the prime decomposition of the ideal.

The present paper is organized as follows. In section 2, we prove the essential normality if the zero variety is simple. Section 3 is devoted to general cases. In Section 4, we study the boundary representation.

\section{Restriction of $H^2(\mathbb D^d)$ to the simple homogenous distinguished variety}

~~~~~~~Let $V_\theta$ denote the simple homogenous distinguished variety of $\mathbb D^d$ through $(\theta_1,\ldots,\theta_d)$, namely
\begin{eqnarray}\label{Def:V_theta}
    V_\theta=\{(\theta_1z,\ldots,\theta_dz):z\in\mathbb{D}\}
\end{eqnarray}
for some fixed $\theta=(\theta_1,\ldots,\theta_d)\in\mathbb{T}^d$. Denote by $J_\theta=I(V_\theta)$ the prime ideal of $\poly$ with variety $V_\theta$. This section is devoted to prove $1$-essential normality of the quotient modules $\mathcal{N}_\theta=[J_\theta^N]^\bot$ of $H^2(\mathbb{D}^d)$, where $N\geq1$ is any integer. Without loss of generality, we assume $\theta=(1,\ldots,1)$ throughout this section.

As in \cite{FR2}, define the restriction map $r:H^2(\mathbb{D}^d)\to\Hol(\mathbb{D})$ such that $(rf)(z)=f(z,\ldots,z)$ for $z\in\mathbb{D}$, then $[J]=\ker r$.

Denote $\omega=e^{\frac{2\pi}{d}i}$ the $d$-th primitive root of unit, and define linear polynomials
$$w_i(z):=\sum_{j=1}^d\omega^{(i-1)(j-1)}z_j,~i=1,\ldots,d.$$
Set $w=(w_1,\ldots,w_d)$. Obviously $\langle w_i,w_j\rangle=0$ for $i\neq j$, and $||w_i||^2=d$ for $i=1,\ldots,d$. Evidently $r$ maps $w_1$ to $d\cdot z$ and $w_i$ to $0$ for $i=2,\ldots,d.$ Simple division shows that $J$ is precisely the ideal generated by $w_2,\ldots,w_d$, and $J^n$ is generated by
$$J_n:=\span\{w^\alpha:|\alpha|=n,\alpha_1=0\}.$$
Write $\partial=(\frac{\partial}{\partial z_1},\ldots,\frac{\partial}{\partial z_d})$. Differentiation by parts shows $r\partial^\alpha f=0$ for all $f\in J^n$ and $|\alpha|<n$.
By direct computation one obtain
$$r\partial^\alpha M_{w_i}=r[w_i(\partial)z^\alpha](\partial),i=2,\ldots,d,$$
and
$$r\partial^\alpha M_{w_1}=r[w_1(\partial)z^\alpha](\partial)+r(w_1)r\partial^\alpha.$$
Then for $p\in\poly$ we have
\begin{equation}\label{eq:derivitive1}
    rp(\partial)M_w^\beta=r[w^\beta(\partial)p](\partial)
\end{equation}
and
\begin{equation}\label{eq:derivitive2}
    rp(\partial)M_hM_\omega^\beta=r[h(\partial)w^\beta(\partial)p](\partial)+r(h)[w^\beta(\partial)p](\partial)
\end{equation}
when $\beta_1=0$ and $h\in\poly$ is linear.

We can check that the operator $\partial^{\alpha*}$ maps constant $1$ to $\alpha!z^\alpha$. Therefore given any polynomial $g\in\poly$, we can find a unique $p_g\in\poly$ making $p_g(\partial)^*1=g$. Set $r_g=rp_g(\partial)$. For $g=\sum_{|\gamma|=n}c_\gamma z^\gamma$ one can verify
\begin{equation}\label{eq:polynomial}
    p_g=\sum_{|\gamma|=n}\bar{c}_\gamma\frac{z^\gamma}{\gamma!}.
\end{equation}
For $f=\sum_{|\gamma|=n}d_\gamma z^\gamma$ we have
\begin{equation}
    f(\partial)p_g=\sum_{|\gamma|=n}\frac{d_\gamma\bar{c}_\gamma}{\gamma!}\partial^\gamma z^\gamma=\sum_{|\gamma|=|\alpha|}d_\gamma\bar{c}_\gamma=\langle f,g\rangle.
\end{equation}
Then by (\ref{eq:derivitive1})
\begin{equation}\label{eq:derivitive2}
    r_gM_f=rp_g(\partial)M_f=r[f(\partial)p_g](\partial)=\langle f,g\rangle r,~\forall f,g\in J_n.
\end{equation}
The following proposition generalizes proposition 1 of \cite{FR2} to the multi-variable case.
\begin{prop}\label{prop:characterization}
    Let $h\in\poly$, then $h\in J^N$ if and only if $r_{w^\alpha}h=0$ whenever $|\alpha|<N$ and $\alpha_1=0$. Consequently
    $$\mathcal{N}=\overline{\span}\{(\ker R_{w^\alpha})^\bot:|\alpha|<N,\alpha_1=0\}.$$
\end{prop}
\begin{proof}
    By former discussion we have $r_{w^\alpha}[J^N]=0$ for $|\alpha|<N$. Conversely assume $h\in\poly$ and $r_{w^\alpha}h=0$ whenever $|\alpha|<N$ and $\alpha_1=0$, and we shall prove $h\in J^N$ by induction. Suppose $N=1$, since $rh=0$ then we have $h\in[w_2,\ldots,w_d]=J$, and the proposition is proved in this case. Assume the proposition is proved for $N=1,\ldots,n$, and suppose $r_{w^\alpha}h=0$ whenever $|\alpha|<n+1$ and $\alpha_1=0$. By the assumption for induction, $h\in J^n$. Denote $\mathcal{J}_n=\{\alpha\in\mathbb{Z}_+^d:|\alpha|=n\mbox{ and }\alpha_1=0\}$, and then we can write $h=\sum_{\beta\in\mathcal{J}_n}w^\beta h_\beta$ where each $h_\beta$ is a polynomial. For $\alpha\in\mathcal{J}_n$ we have
    \begin{eqnarray}\label{eq:independence}
        r_{w^\alpha}h&=&\sum_{\beta\in\mathcal{J}_n}r_{w^\alpha}M_w^\beta h_\beta\\
        &=&\sum_{\beta\in\mathcal{J}_n}\langle w^\beta,w^\alpha\rangle rh_\beta\notag\\
        &=&0,\notag
    \end{eqnarray}
    where the second equality is deduced by (\ref{eq:derivitive2}).
    Since each homogeneous polynomial of degree $n$ is a linear combination of $\{w^\alpha:|\alpha|=n\}$, the later forms a linearly independent system. This implies the invertibility of the matrix $(\langle w^\beta,w^\alpha\rangle)_{\alpha,\beta\in\mathcal{J}_n}$, and then each $rh_\beta$ should be zero by (\ref{eq:independence}). Therefore $h_\beta\in J$ and $h\in J^{n+1}$. By induction, the proposition holds for all natural numbers $N$.
\end{proof}
\begin{rem}
    By linearity, proposition \ref{prop:characterization} actually states that $h\in J^n$ if and only if $r_gh=0$ for every $g\in J_0\cup J_1\ldots\cup J_{n-1}$.
\end{rem}
Similar to \cite{FR2} for $\lambda\in\mathbb{D}$ we write
$$K_\lambda^\otimes(z)=\prod_{i=1}^dK_\lambda(z_i),z\in\mathbb{D}^d$$
and
$$k_\lambda^\otimes(z)=\prod_{i=1}^dk_\lambda(z_i),z\in\mathbb{D}^d,$$
where $k_\lambda=\frac{K_\lambda}{||K_\lambda||}$ is the normalized reproducing kernel for $H^2(\mathbb{D})$.

For $f\in H^2(\mathbb{D}^d)$ and $g\in\poly$, we have
\begin{eqnarray*}
    r_gf(\lambda)&=&r[p_g(\partial)f](\lambda)\\
    &=&p_g(\partial)f(\lambda,\ldots,\lambda)\\
    &=&\langle p_g(\partial)f,K_\lambda^\otimes\rangle\\
    &=&\langle f,p_g(\partial)^*K_\lambda^\otimes\rangle
\end{eqnarray*}
Therefore
\begin{equation}\label{eq:kernel1}
    (\ker r_g)^\bot=\overline{\span}\{p_g(\partial)^*K_\lambda^\otimes:\lambda\in\mathbb{D}\}.
\end{equation}
By equality (\ref{eq:polynomial}), for $w^\alpha=\sum_{|\gamma|=n}c_\gamma z^\gamma\in J_n$ we have
\begin{eqnarray}\label{eq:kernel2}
    p_{w^\alpha}(\partial)^*K_\lambda^\otimes(z)&=&\sum_{|\gamma|=n}c_\gamma\frac{1}{\gamma!}\partial^{\gamma*}K_\lambda^\otimes(z)\\
    &=&\sum_{|\gamma|=n}\left(c_\gamma\prod_{i=1}^d z^{\gamma_i}K_\lambda(z_i)^{\gamma_i+1}\right)\notag\\
    &=&w^\alpha(\psi_\lambda(z))K_\lambda^\otimes(z),\notag
\end{eqnarray}
where $\psi_\lambda(z):=\left(\frac{z_1}{1-\bar{\lambda}z_1},\ldots,\frac{z_d}{1-\bar{\lambda}z_d}\right),~\lambda\in\mathbb{D}$. Obviously
$$r(w_i\circ\psi_\lambda)(\mu)=0,~\forall\mu\in\mathbb{D},i=2,\ldots,d.$$
Therefore $w_i\circ\psi_\lambda\in [J]$ and $w^\alpha\circ\psi_\lambda\in[J^n]$. By linearity, $g\circ\psi_\lambda\in[J^n]$ whenever $g\in J_n$. Combining this fact with (\ref{eq:kernel1}) and (\ref{eq:kernel2}) we obtain the following lemma.
\begin{lem}\label{lem:coker}
    $(\ker r_g)^\bot=\overline{\span}\{g\circ\psi_\lambda\cdot K_\lambda^\otimes:\lambda\in\mathbb{D}\}\subset[J^n]$ for $g\in J_n$.
\end{lem}
Denote $H_n=\span\{(\ker r_g)^\bot :g\in J_n\}$, then by proposition \ref{prop:characterization} and its remark
$$\mathcal{N}=\overline{\span}\{H_n:n=0,1,\ldots,N-1\}.$$
Moreover, we have the following proposition, which gives the precise structure of $\cal N$.
\begin{prop}\label{prop:deg}
    $H_m\bot H_n$ whenever $m\neq n$. As a consequence, $\mathcal{N}=\bigoplus_{n=0}^{N-1}H_n$.
\end{prop}
\begin{proof}
    By the previous lemma $H_n\subset[J^n]$, and since $[J^n]^\bot=\overline{\span}\{H_m:0\leq m\leq n-1\}$ for each $n$, we have $H_n\bot H_m$ provided $m<n$.
\end{proof}
For $\lambda\in\mathbb{D}$, denote the Mobius transform on $\mathbb{D}^d$ by $$\varphi_\lambda:z\mapsto\left(\frac{\lambda-z_1}{1-\bar{\lambda}z_1},\ldots,\frac{\lambda-z_d}{1-\bar{\lambda}z_d}\right),z\in\mathbb{D}^d.$$
It is well-known that the linear map $V:f\mapsto f\circ\varphi_\lambda\cdot k_\lambda^\otimes$ is a unitary on $H^2(\mathbb{D}^d)$, such that $V^2=\Id$.

Suppose $w^\alpha,w^\beta\in J_n,\lambda\in\mathbb{D}$, then by (\ref{eq:kernel2}) we have
\begin{eqnarray*}
    \langle p_{w^\alpha}(\partial)^*k_\lambda^\otimes,p_{w^\beta}(\partial)^*k_\lambda^\otimes\rangle&=&\langle w^\alpha\circ\psi_\lambda\cdot k_\lambda^\otimes,w^\beta\circ\psi_\lambda\cdot k_\lambda^\otimes\rangle\\
    &=&\langle w^\alpha\circ\psi_\lambda\circ\varphi_\lambda,w^\beta\circ\psi_\lambda\circ\varphi_\lambda\rangle\\
    &=&\left\langle w^\alpha\left(\frac{\lambda-z_1}{1-|\lambda|^2},\ldots,\frac{\lambda-z_d}{1-|\lambda|^2}\right),w^\beta\left(\frac{\lambda-z_1}{1-|\lambda|^2},\ldots,\frac{\lambda-z_d}{1-|\lambda|^2}\right)\right\rangle\\
    &=&(1-|\lambda|^2)^{-2n}\langle w^\alpha(\lambda-z_1,\ldots,\lambda-z_d),w^\beta(\lambda-z_1,\ldots,\lambda-z_d)\rangle\\
    &=&(1-|\lambda|^2)^{-2n}\langle w^\alpha(-z_1,\ldots,-z_d),w^\beta(-z_1,\ldots,-z_d)\rangle\\
    &=&(1-|\lambda|^2)^{-2n}\langle w^\alpha,w^\beta\rangle.
\end{eqnarray*}
This induces
$$\langle p_f(\partial)^*K_\lambda^\otimes,p_g(\partial)^*K_\lambda^\otimes\rangle=(1-|\lambda|^2)^{-d-2n}\langle f,g\rangle,~\forall\lambda\in\mathbb{D}$$
whenever $f,g\in J_n$. Notice that the function $F(\lambda,\mu)=\langle p_f(\partial)^*K_\lambda^\otimes,p_g(\partial)^*K_\mu^\otimes\rangle$ is analytic on $\mu$ and co-analytic on $\lambda$.  By Proposition 1 of \cite{Eng},
\begin{equation}\label{eq:innerproduct}
    \langle p_f(\partial)^*K_\lambda^\otimes,p_g(\partial)^*K_\mu^\otimes\rangle=(1-\langle\lambda,\mu\rangle)^{-d-2n}\langle f,g\rangle,~\forall\lambda,\mu\in\mathbb{D}.
\end{equation}

Next, let $\cA_{d+2n-2}$ denote the weighted Bergman space on the unit disc with the reproducing kernel $K_{\lambda}^{(d+2n-2)}$.
For homogeneous $g\in J_n$ of $||g||=1$, by (\ref{eq:innerproduct}) we can define an isometry $R_g:(\ker r_g)^\bot\to\mathcal{A}_{d+2n-2}$, satisfying $R_g(p_g(\partial)^*K_\lambda^\otimes)=K_\lambda^{(d+2n-2)}$. Then it holds for $\lambda,\mu\in\mathbb{D}$ that
\begin{eqnarray*}
    r_g[p_g(\partial)^*K_\lambda^\otimes](\mu)&=&\langle p_g(\partial)p_g(\partial)^*K_\lambda^\otimes,K_\mu^\otimes\rangle\\
    &=&\langle p_g(\partial)^*K_\lambda^\otimes,p_g(\partial)^*K_\mu^\otimes\rangle\\
    &=&(1-\bar{\lambda}\mu)^{-d-2n}~~~~~~~~~~~~~~~(\text{by }\ref{eq:innerproduct})\\
    &=&\langle K_\lambda^{(d+2n-2)},K_\mu^{(d+2n-2)}\rangle\\
    &=&R_g[p_g(\partial)^*K_\lambda^\otimes](\mu),
\end{eqnarray*}
which implies that $R_g$ is actually the restriction of $r_g$ on $(\ker r_g)^\bot$. Namely we have
\begin{equation}\label{eq:isometry}
    r_g[p_g(\partial)^*K_\lambda^\otimes]=K_\lambda^{(d+2n-2)},~\forall~\lambda\in\mathbb{D},
\end{equation}
and that $r_g$ is an isometry on $(\ker r_g)^\bot$.

For each subspace $J_n\subset H^2(\mathbb{D}^d)$ we choose an orthonormal basis $\mathcal{B}_n=\{f_j^{(n)}:j=1,\ldots,\binom{d+n-2}{d-2}\}$, and denote $\mathcal{B}:=\bigcup_{n=0}^{N-1}\mathcal{B}_n=\{g_j:j=1,\ldots,\binom{N+d-2}{d-1}\}$. By (\ref{eq:innerproduct}), $(\ker r_{f_i^{(n)}})^\bot$ is orthogonal to $(\ker r_{f_j^{(n)}})^\bot$ whenever $i\neq j$, and therefore $H_n=\bigoplus_{g\in\mathcal{B}_n}(\ker r_g)^\bot$. Then by proposition \ref{prop:deg} we have $\mathcal{N}=\bigoplus_{g\in\mathcal{B}}(\ker r_g)^\bot$. Define a linear mapping $U:\mathcal{N}\to\bigoplus_{n=0}^{N-1}\bigoplus_{g\in\mathcal{B}_n}\mathcal{A}_{d+2n-2}$ by
\begin{equation}\label{eq:U}
    Uh=(r_gh)_{g\in\mathcal{B}}.
\end{equation}
Since each $r_g$ is an isometry on $(\ker r_g)^\bot$, $U$ is a unitary.
\begin{prop}\label{prop:multiplier}
    For $g\in\mathcal{B}_n$ and $f\in\poly$, it holds that
    $$r_g(fh)=r(f)r_g(h),~\forall h\in(\ker r_g)^\bot.$$
\end{prop}
\begin{proof}
    For $\lambda\in\mathbb{D}$ we have
    \begin{eqnarray*}
        \langle g\circ\psi_\lambda\cdot fk_\lambda^\otimes,g\circ\psi_\lambda\cdot k_\lambda^\otimes\rangle&=&\langle f\circ\varphi_\lambda\cdot g\circ\psi_\lambda\circ\varphi_\lambda,g\circ\psi_\lambda\circ\varphi_\lambda\rangle\\
        &=&\left\langle f\circ\varphi_\lambda\cdot g\left(\frac{\lambda-z_1}{1-|\lambda|^2},\ldots,\frac{\lambda-z_d}{1-|\lambda|^2}\right),g\left(\frac{\lambda-z_1}{1-|\lambda|^2},\ldots,\frac{\lambda-z_d}{1-|\lambda|^2}\right)\right\rangle\\
        &=&(1-|\lambda|^2)^{-2n}\langle f\circ\varphi_\lambda\cdot g(\lambda-z_1,\ldots,\lambda-z_d),g(\lambda-z_1,\ldots,\lambda-z_d)\rangle\\
        &=&(1-|\lambda|^2)^{-2n}\langle f\circ\varphi_\lambda\cdot g(-z_1,\ldots,-z_d),g(-z_1,\ldots,-z_d)\rangle\\
        &=&(1-|\lambda|^2)^{-2n}\langle f\circ\varphi_\lambda\cdot g,g\rangle\\
        &=&(1-|\lambda|^2)^{-2n}\langle f\circ\varphi_\lambda(0)g,g\rangle\\
        &=&(1-|\lambda|^2)^{-2n}r(f)(\lambda),
    \end{eqnarray*}
    and then by equality (\ref{eq:isometry})
    \begin{eqnarray*}
        \langle r_g(fr_g^*K_\lambda^{(d+2n-2)}),K_\lambda^{(d+2n-2)}\rangle&=&\langle r_g(g\circ\psi_\lambda\cdot fK_\lambda^\otimes),r_g(g\circ\psi_\lambda\cdot K_\lambda^\otimes)\rangle\\
        &=&\langle g\circ\psi_\lambda\cdot fK_\lambda^\otimes,g\circ\psi_\lambda\cdot K_\lambda^\otimes\rangle\\
        &=&(1-|\lambda|^2)^{-d-2n}r(f)(\lambda)\\
        &=&\langle r(f)K_\lambda^{(d+2n-2)},K_\lambda^{(d+2n-2)}\rangle.
    \end{eqnarray*}
    Again by Proposition 1 in \cite{Eng},
    $$\langle r_g(fr_g^*K_\lambda^{(d+2n-2)}),K_\mu^{(d+2n-2)}\rangle=\langle r(f)K_\lambda^{(d+2n-2)},K_\mu^{(d+2n-2)}\rangle$$
    whenever $\lambda,\mu\in\mathbb{D}$, and hence
    $$r_g(fr_g^*K_\lambda^{(d+2n-2)})=r(f)r_gr_g^*K_\lambda^{(d+2n-2)},~\forall~\lambda\in\mathbb{D}.$$
    Since $\{r_g^*K_\lambda^{(d+2n-2)}:\lambda\in\mathbb{D}\}$ is dense in $(\ker r_g)^\bot$, the proposition is proved.
\end{proof}
Denote by $P_g(g\in\mathcal{B})$ the orthogonal projection to $(\ker r_g)^\bot$, and $A_{i}^{(f,g)}:=P_fS_{z_i}P_g$. Then we have $S_{z_i}=\sum_{f,g\in\mathcal{B}}A_{i}^{(f,g)}$. By proposition \ref{prop:multiplier}, $A_{i}^{(g,g)}$ is unitarily equivalent to $M_z$ acting on $\mathcal{A}_{d+2n-2}$, which is $1$-essentially normal with $\tr[M_z^*,M_z]=1$. Therefore $[A_{i}^{(g,g)*},A_{i}^{(g,g)}]$ belongs to the trace class and is of trace $1$. Next proposition concerns the compactness of $A_{i}^{(f,g)}$ where $f\neq g$.
\begin{prop}\label{prop:shift}
    Given $f\in\mathcal{B}_m$ and $g\in\mathcal{B}_n$, then
    \begin{itemize}
        \item[(a)]$A_{i}^{(f,g)}=0$ if $m\leq n$ and $f\neq g$;
        \item[(b)]$A_{i}^{(f,g)}\in\mathcal{L}^2$ if $m>n$.
    \end{itemize}
\end{prop}
\begin{proof}
    For $\lambda\in\mathbb{D}$ we have
    \begin{eqnarray}\label{eq:changeshift}
        \langle g\circ\psi_\lambda\cdot z_ik_\lambda^\otimes,f\circ\psi_\lambda\cdot k_\lambda^\otimes\rangle&=&\left\langle\frac{\lambda-z_i}{1-\bar{\lambda}z_i}\cdot g\circ\psi_\lambda\circ\varphi_\lambda,f\circ\psi_\lambda\circ\varphi_\lambda\right\rangle\notag\\
        &=&\left\langle\frac{\lambda-z_i}{1-\bar{\lambda}z_i}g\left(\frac{\lambda-z_1}{1-|\lambda|^2},\ldots,\frac{\lambda-z_d}{1-|\lambda|^2}\right),f\left(\frac{\lambda-z_1}{1-|\lambda|^2},\ldots,\frac{\lambda-z_d}{1-|\lambda|^2}\right)\right\rangle\notag\\
        &=&(1-|\lambda|^2)^{-n-m}\left\langle\frac{\lambda-z_i}{1-\bar{\lambda}z_i}g(\lambda-z_1,\ldots,\lambda-z_d),f(\lambda-z_1,\ldots,\lambda-z_d)\right\rangle\notag\\
        &=&(1-|\lambda|^2)^{-n-m}\left\langle\frac{\lambda-z_i}{1-\bar{\lambda}z_i}g(-z_1,\ldots,-z_d),f(-z_1,\ldots,-z_d)\right\rangle\notag\\
        &=&(1-|\lambda|^2)^{-n-m}\left\langle\frac{\lambda-z_i}{1-\bar{\lambda}z_i}g,f\right\rangle.\notag\\
    \end{eqnarray}
    In the cases $m\leq n$ and $f\bot g$ and $\deg(f)\leq \deg(g)$. Notice that $\frac{\lambda-z_i}{1-\bar{\lambda}z_i}g=(\lambda-z_i)\sum\limits_{n=0}^\infty (\bar{\lambda}z_i)^n g$. Since for any $n\geq 0$, $z_i^n g\perp f$, and we have $\langle\frac{\lambda-z_i}{1-\bar{\lambda}z_i}g,f\rangle=0$, and therefore
    $$\langle z_ir_g^*k_\lambda^{(d+2n-2)},r_f^*k_\lambda^{(d+2m-2)}\rangle=0$$
    inducing
    $$\langle z_ir_g^*k_\lambda^{(d+2n-2)},r_f^*k_\mu^{(d+2m-2)}\rangle=0,~\forall\lambda,\mu\in\mathbb{D}$$
    by Proposition 1 of \cite{Eng}. This equality together with Lemma \ref{lem:coker} shows
    $$A_{i}^{(f,g)}=P_fM_{z_i}P_g=0$$
    in these cases.

    In the case $m>n$, by the expansion
    $$\frac{\lambda-z_i}{1-\bar{\lambda}z_i}=\lambda-(1-|\lambda|^2)\sum_{k=1}^\infty\bar{\lambda}^{k-1}z_i^k$$
    we find by equality (\ref{eq:changeshift})
    $$\langle z_ir_g^*k_\lambda^{(d+2n-2)},r_f^*k_\lambda^{(d+2m-2)}\rangle=-\bar{\lambda}^{m-n-1}(1-|\lambda|^2)^{1-n-m}\langle z_i^{m-n}g,f\rangle$$
    and therefore
    $$\langle z_ir_g^*K_\lambda^{(d+2n-2)},r_f^*K_\lambda^{(d+2m-2)}\rangle=-\bar{\lambda}^{m-n-1}(1-|\lambda|^2)^{1-d-n-m}\langle z_i^{m-n}g,f\rangle.$$
    An application of Proposition 1 of \cite{Eng} shows
    $$\langle z_ir_g^*K_\lambda^{(d+2n-2)},r_f^*K_\mu^{(d+2m-2)}\rangle=-\bar{\lambda}^{m-n-1}(1-\mu\bar{\lambda})^{1-d-n-m}\langle z_i^{m-n}g,f\rangle.$$
    Denote by $\tilde{A}_{i}^{(f,g)}:=r_fA_{i}^{(f,g)}r_g^*:\mathcal{A}_{d+2n-2}\to\mathcal{A}_{d+2m-2}$, then we have
    \begin{equation}\label{eq:afg}
        \tilde{A}_{i}^{(f,g)}K_\lambda^{(d+2n-2)}=-\bar{\lambda}^{m-n-1}(1-z\bar{\lambda})^{1-d-n-m}\langle z_i^{m-n}g,f\rangle.
    \end{equation}
    By comparing the coefficients of $\bar{\lambda}^k(k\geq m-n-1)$ we find
    $$\tilde{A}_{i}^{(f,g)}\binom{d+2n+k-1}{d+2n-1}z^k=-\binom{d+2n+k-1}{d+m+n-2}z^{k-m+n+1}\langle z_i^{m-n}g,f\rangle,$$
    and therefore
    $$\tilde{A}_{i}^{(f,g)}z^k=-\binom{d+2n+k-1}{d+m+n-2}\binom{d+2n+k-1}{d+2n-1}^{-1}\langle z_i^{m-n}g,f\rangle z^{k-m+n+1}.$$
    Then from
    $$\frac{||z^{n-m+k+1}||_{d+2m-2}^2}{||z^k||_{d+2n-2}^2}=\binom{d+2n+k-1}{d+2n-1}\binom{d+n+m+k}{d+2m-1}^{-1}$$
    we obtain
    \begin{equation}\label{eq:shift}
        \tilde{A}_{i}^{(f,g)}\frac{z^k}{||z^k||_{d+2n-2}}=\langle z_i^{m-n}g,f\rangle a_{m,n}(k)\frac{z^{k-m+n+1}}{||z^{k-m+n+1}||_{d+2m-2}},
    \end{equation}
    where
    $$a_{m,n}(k)=-\binom{d+2n+k-1}{d+m+n-2}\binom{d+2n+k-1}{d+2n-1}^{-1/2}\binom{d+n+m+k}{d+2m-1}^{-1/2}.$$
    Since $a_{m,n}(k)=O(k^{-1})$ as $k\to+\infty$, we find $\tilde{A}_{i}^{(f,g)}\in\mathcal{L}^2$.
\end{proof}
\begin{cor}\label{cor:isometry}
    There is an isometry $S\in B(\mathcal{N})$ which is unitarily equivalent to a $\binom{N+d-2}{d-1}$-shift and compact operators $K_i\in\mathcal{L}^{(1,\infty)}$ such that $S_{z_i}=S-K_i$. Moreover, there is a constant $c$ such that if $h\in\mathcal{N}$ is homogeneous and of degree $k$, then $||K_ih||\leq c(k+1)^{-1}||h||$.
\end{cor}
\begin{proof}
    By proposition \ref{prop:multiplier}, for each $f\in\mathcal{B}_m$
    $$\tilde{A}_{i}^{(f,f)}\frac{z^k}{||z^k||_{d+2m-2}}=\frac{||z^{k+1}||_{d+2m-2}}{||z^k||_{d+2m-2}}\frac{z^{k+1}}{||z^{k+1}||_{d+2m-2}}=\left(\frac{k+1}{d+2m+k}\right)^{1/2}\frac{z^{k+1}}{||z^{k+1}||_{d+2m-2}},$$
    and then
    \begin{eqnarray}\label{eq:shift2}
        &&\tilde{A}_{i}^{(f,f)}\frac{z^k}{||z^k||_{d+2m-2}}-\frac{z^{k+1}}{||z^{k+1}||_{d+2m-2}}\\
        &=&\left[\left(\frac{k+1}{d+2m+k}\right)^{1/2}-1\right]\frac{z^{k+1}}{||z^{k+1}||_{d+2m-2}}\notag\\
        &=&-\frac{d+2m-1}{(d+2m+k)\left[1+\left(\frac{k+1}{d+2m+k}\right)^{1/2}\right]}\frac{z^{k+1}}{||z^{k+1}||_{d+2m-2}}\notag
    \end{eqnarray}
    Denote by $\tilde{S}$ the isometry on $\bigoplus_{m=0}^{N-1}\bigoplus_{f\in\mathcal{B}_m}r_f\mathcal{N}$ that maps each $\frac{z^k}{||z^k||_{d+2m-2}}\in r_f\mathcal{N}$ to $\frac{z^{k+1}}{||z^{k+1}||_{d+2m-2}}$, then by equalities (\ref{eq:shift}) and (\ref{eq:shift2}),
    $$\tilde{K}_i=\tilde{S}-US_{z_i}U^*\in\mathcal{L}^{(1,\infty)}.$$
    The corollary holds for $S=U^*\tilde{S}U$ and $K_i=U^*\tilde{K}_iU$.
\end{proof}
\begin{rem}\label{rem:shif}
    In the case $\theta\neq(1,\ldots,1)$, define linear transform
    $$L_\theta:\mathbb{C}^d\to\mathbb{C}^d, (z_1,\ldots,z_d)\mapsto(\theta_1z_1,\ldots,\theta_dz_d).$$
    Then $U_\theta:H^2(\mathbb{D}^d)\to H^2(\mathbb{D}^d),f\mapsto f\circ L_\theta$ maps $\mathcal{N}_\theta=[J_\theta^N]^\bot$ unitarily onto $\mathcal{N}=[J^N]^\bot$. By the corollary, the is an isometry $S\in B(\mathcal{N})$ and compact operators $K_i$ such that $S_{z_i}=S-K_i$ for $1\leq i\leq d$. Therefore
    \begin{equation}\label{eq:rot}
        S_{z_i}\circ L_\theta=S\circ L_\theta-K_i\circ L_\theta.
    \end{equation}
    Denote by $S_\theta=S\circ L_\theta$, then $S_\theta$ is an isometry on $\mathcal{N}_\theta$, and $K_i\circ L_\theta$ is compact. Then equality (\ref{eq:rot}) actually shows that the contraction of $M_{z_i}$ on $\mathcal{N}_\theta$ is the sum of the isometry $S_\theta$ and some compact operator.
\end{rem}
\begin{cor}\label{cor:upright}
    $P_\mathcal{N}^\bot M_{z_i}P_\mathcal{N}\in\mathcal{L}^{(2,\infty)}$.
\end{cor}
\begin{proof}
    By corollary \ref{cor:isometry}, $S_{z_i}=S-K_i$ where $S$ is an isometry and there is a constant $c$ such that $||K_ih||\leq c(\deg h)^{-1}||h||$ for homogeneous $h\in\mathcal{N}$. Suppose $h\in\mathcal{N}$ is homogeneous, then
    \begin{eqnarray*}
        ||P_\mathcal{N}^\bot M_{z_i}P_\mathcal{N}h||^2&=&\langle P_\mathcal{N}^\bot M_{z_i}h,M_{z_i}h\rangle\\
        &=&||M_{z_i}h||^2-\langle S_{z_i}h,S_{z_i}h\rangle\\
        &=&||h||^2-\langle(S-K_i)h,S-K_ih\rangle\\
        &=&\langle Sh,K_ih\rangle+\langle K_ih,Sh\rangle-\langle K_ih,K_ih\rangle\\
        &\leq&2||K_ih||\\
        &\leq&2c(\deg h)^{-1}||h||.
    \end{eqnarray*}
    Therefore we conclude $P_\mathcal{N}^\bot M_{z_i}P_\mathcal{N}\in\mathcal{L}^{(2,\infty)}$ as desired.
\end{proof}
\begin{thm}\label{thm:std}
    $\mathcal{N}$ is $1$-essentially normal.
\end{thm}
\begin{proof}
    Let $U$ be the unitary defined in (\ref{eq:U}). By definition we need to prove for $i,j=1,\ldots,d$ that
    \begin{eqnarray}
        U[S_{z_i}^*,S_{z_j}]U^*=\sum_{f,h,g\in\mathcal{B}}\tilde{A}_i^{(h,f)*}\tilde{A}_j^{(h,g)}-\tilde{A}_j^{(f,h)}\tilde{A}_i^{(g,h)*}\in\mathcal{L}^1.
    \end{eqnarray}
    By proposition \ref{prop:shift}, the terms for which $h\neq f$ and $h\neq g$ must belong to $\mathcal{L}^1$. Then it is sufficient to prove for $f,g\in\mathcal{B}$ that
    \begin{equation}\label{eq:shift1}
        \tilde{A}_{j}^{(f,f)*}\tilde{A}_{i}^{(f,g)}-\tilde{A}_{i}^{(f,g)}\tilde{A}_{j}^{(g,g)*}\in\mathcal{L}^1
    \end{equation}
    and
    $$\tilde{A}_i^{(g,f)*}\tilde{A}_j^{(g,g)}-\tilde{A}_j^{(f,f)}\tilde{A}_i^{(g,f)*}\in\mathcal{L}^1.$$
    Suppose $f\in\mathcal{B}_m$ and $g\in\mathcal{B}_n$. By symmetry we can assume $m\geq n$, and then by proposition \ref{prop:shift}
    $$\tilde{A}_i^{(g,f)*}\tilde{A}_j^{(g,g)}-\tilde{A}_j^{(f,f)}\tilde{A}_i^{(g,f)*}=0$$
    if $f\neq g$, and
    $$\tilde{A}_j^{(f,f)*}\tilde{A}_{i}^{(f,g)}-\tilde{A}_i^{(f,g)}\tilde{A}_j^{(g,g)*}=0$$
    if $m=n$ and $f\neq g$. When $f=g$ then $\tilde{A}_i^{(f,f)}$ is the shift $M_z$ acting on $A_{d+2n-2}$, and therefore $[\tilde{A}_i^{(f,f)*},\tilde{A}_j^{(f,f)}]\in\mathcal{L}^1$.
    Then it suffices to prove (\ref{eq:shift1}) for $m>n$.

    It is routine to check
    $$\tilde{A}_{i}^{(f,f)*}\frac{z^k}{||z^k||}=\sqrt{\frac{k}{d+2m+k-1}}\frac{z^{k-1}}{||z^{k-1}||},~\forall k\in\mathbb{Z}_+$$
    and
    $$\tilde{A}_{j}^{(g,g)*}\frac{z^k}{||z^k||}=\sqrt{\frac{k}{d+2n+k-1}}\frac{z^{k-1}}{||z^{k-1}||},~\forall k\in\mathbb{Z}_+.$$
    Then by (\ref{eq:shift})
    \begin{eqnarray*}
        &&(\tilde{A}_{j}^{(f,f)*}\tilde{A}_{i}^{(f,g)}-\tilde{A}_{i}^{(f,g)}\tilde{A}_{j}^{(g,g)*})\frac{z^k}{||z^k||_{d+2n-2}}\\
        &=&\langle z_i^{m-n}g,f\rangle\left(a_{m,n}(k)\sqrt{\frac{k-m+n+1}{d+m+n+k}}-a_{m,n}(k-1)\sqrt{\frac{k}{d+2n+k-1}}\right)\frac{z^{k-m+n}}{||z^{k-m+n}||_{d+2m-2}}\\
        &=&\langle z_i^{m-n}g,f\rangle b_{m,n}(k)\frac{z^{k-m+n}}{||z^{k-m+n}||_{d+2m-2}}
    \end{eqnarray*}
    where $b_{m,n}(k)=O(k^{-2})$ as $k\to+\infty$. Therefore (\ref{eq:shift1}) holds for $m>n$, completing the proof of the theorem.
\end{proof}
To find the traces of commutators of multiplication operators, we need the following lemma.
\begin{lem}\label{lem:trace}
    $\tr[S_{f_1}^*,S_{f_2}]=0$ for $f_1\in J$ and $f_2\in\poly$.
\end{lem}
\begin{proof}
    $[S_{f_1}^*,S_{f_2}]$ belongs to the trace class by Theorem \ref{thm:std}. We have
    \begin{equation}\label{eq:trace}
        [S_{f_1}^*,S_{f_2}]=\sum_{g,h,l\in\mathcal{B}}P_gS_{f_1}^*P_lS_{f_2}P_h-P_gS_{f_2}P_lS_{f_1}^*P_h.
    \end{equation}
    If $g,h,l$ are different from each other, proposition \ref{prop:shift} shows that both $P_gS_{f_1}^*P_lS_{f_2}P_h$ and $P_gS_{f_2}P_lS_{f_1}^*P_h$ belong to the trace class. Then the argument before proposition \ref{prop:multiplier} shows $P_g\bot P_h$ and we have
    \begin{equation}\label{eq:different}
        \tr P_gS_{f_1}^*P_lS_{f_2}P_h=\tr P_gS_{f_2}P_lS_{f_1}^*P_h=0.
    \end{equation}
    When $g\neq h$, the proof of Theorem \ref{thm:std} shows
    $$P_gS_{f_1}^*P_gS_{f_2}P_h-P_gS_{f_2}P_hS_{f_1}^*P_h\in\mathcal{L}^1$$
    and
    $$P_gS_{f_1}^*P_gS_{f_2}P_h-P_gS_{f_2}P_hS_{f_1}^*P_h\in\mathcal{L}^1,$$
    then since $P_g\bot P_h$ we find
    $$\tr(P_gS_{f_1}^*P_gS_{f_2}P_h-P_gS_{f_2}P_hS_{f_1}^*P_h)=\tr(P_gS_{f_1}^*P_gS_{f_2}P_h-P_gS_{f_2}P_hS_{f_1}^*P_h)=0.$$
    This equality together with (\ref{eq:different}) shows
    $$\sum_{g,h,l\in\mathcal{B},g\neq h}\tr(P_gS_{f_1}^*P_lS_{f_2}P_h-P_gS_{f_2}P_lS_{f_1}^*P_h)=0.$$
    Since $f_1\in J$, $S_{f_1}$ maps each $J_n$ into $J^{n+1}$, and then we have
    \begin{eqnarray*}
        \tr[S_{f_1}^*,S_{f_2}]&=&\tr\sum_{g,l\in\mathcal{B}}P_gS_{f_1}^*P_lS_{f_2}P_g-P_gS_{f_2}P_lS_{f_1}^*P_g\\
        &=&\tr\sum_{g,l\in\mathcal{B}}P_gS_{f_1}^*P_lS_{f_2}P_g-P_lS_{f_2}P_gS_{f_1}^*P_l\\
        &=&\sum_{g,l\in\mathcal{B},\deg l>\deg g}\tr(P_gS_{f_1}^*P_lS_{f_2}P_g-P_lS_{f_2}P_gS_{f_1}^*P_l)\\
        &=&0,
    \end{eqnarray*}
    where the last equality comes from the fact that both $P_gS_{f_1}^*P_l$ and $P_lS_{f_2}P_g$ belong to $\mathcal{L}^2$.
\end{proof}
Next, we have the following trace formula.
\begin{thm}
    $\tr[S_{f_1}^*,S_{f_2}]=\binom{d+N-2}{d-1}\langle (rf_2)^\prime,(rf_1)^\prime\rangle_{L_a^2(\mathbb D)}$ for $f_1,f_2\in\poly$.
\end{thm}
\begin{proof}
    Denote by $g_i(z)=rf_i(z_1),i=1,2.$ The previous lemma shows
    $$\tr[S_{f_1}^*,S_{f_2}]=\tr[S_{g_1}^*,S_{g_2}].$$
    By the proof of Lemma \ref{lem:trace},
    $$\sum_{g,h,l\in\mathcal{B},g\neq h}\tr(P_gS_{g_1}^*P_lS_{g_2}P_h-P_gS_{g_2}P_lS_{g_1}^*P_h)=0.$$
    When $g=h$ and $l\neq h$, proposition \ref{prop:shift} shows $P_gS_{g_1}^*P_l,P_lS_{g_2}P_h\in\mathcal{L}^2$, and then
    $$\sum_{h,l\in\mathcal{B},l\neq h}\tr(P_hS_{g_1}^*P_lS_{g_2}P_h-P_hS_{g_2}P_lS_{g_1}^*P_h)=0.$$
    Therefore
    \begin{eqnarray*}
        \tr[S_{g_1}^*,S_{g_2}]&=&\tr\sum_{g,h,l\in\mathcal{B}}P_gS_{g_1}^*P_lS_{g_2}P_h-P_gS_{g_2}P_lS_{g_1}^*P_h\\
        &=&\tr\sum_{h\in\mathcal{B}}P_hS_{g_1}^*P_hS_{g_2}P_h-P_hS_{g_2}P_hS_{g_1}^*P_h\\
        &=&\sum_{n=0}^{N-1}\sum_{h\in\mathcal{B}_n}\tr[M_{rf_1}^{(d+2n-2)*},M_{rf_2}^{(d+2n-2)}]\\
        &=&\sum_{n=0}^{N-1}\binom{d+n-2}{d-2}\langle (rf_2)^\prime,(rf_1)^\prime\rangle_{L_a^2(\mathbb D)}~~~~~~~~\text{(by \cite{Fel})}\\
        &=&\binom{d+N-2}{d-1}\langle (rf_2)^\prime,(rf_1)^\prime\rangle_{L_a^2(\mathbb D)}.
    \end{eqnarray*}
\end{proof}
To prove the essential normality of general homogeneous quotient modules, we need the following lemma.
\begin{lem}
    For each $1\leq i\leq d$, $S_{z_i}$ is $\binom{N+d-2}{d-1}$-multicyclic.
\end{lem}
\begin{proof}
    Suppose $f\in\mathcal{N}$ is a homogeneous polynomial of degree $n$. Then we can write $f$ as
    $$f(z)=\sum_{\alpha\in\mathbb{Z}_+^d,\alpha_1=0}c_\alpha w^\alpha z_i^{n-|\alpha|},\forall z\in\mathbb{D}^d,$$
    where the $c_\alpha$'s are coefficients. Since $w^\alpha z_i^{n-|\alpha|}\in[J^N]$ whenever $|\alpha|\geq N$, we find
    $$f=\sum_{\alpha\in\mathcal{B}}c_\alpha S_{z_i}^{n-|\alpha|}w^\alpha.$$
    Then the lemma is proved.
\end{proof}
\begin{thm}\label{thm:primary}
    Suppose $I\subset\poly$ is a homogeneous ideal with $Z(I)=V$. Then the quotient module $\mathcal{M}=[I]^\bot$ is $1$-essentially normal.
\end{thm}
\begin{proof}
    Let $I=I_1\cap I_2$ be the primary decomposition, where $Z(I_1)=V$ and $Z(I_2)=\{0\}$. Let $\{f_1,\ldots,f_n\}$ be a set of generators of $I_1$. By Hilbert's Nullstellensatz, $J=\sqrt{I_1}$ and therefore for each $i=2,\ldots,d$ there is an positive integer $n_i$ such that $w_i^{n_i}\in I_2$. Let $N_1=\sum_{i=2}^d n_i-d+2$ then $J^{N_2}\subset I_1$. Similarly we can find $N_2\in\mathbb{N}$ such that $m^{N_2}\subset I_2$ where $m\subset\poly$ is the maximal ideal generated by $w_1,\ldots,w_d$. Set $N=\max\{N_1,N_2\}$ then $J^N\subset I_1\cap I_2=I$. Denote by $\mathcal{N}=[J^N]^\bot$ then $\mathcal{M}\subset\mathcal{N}$.

    For $i=1,\ldots,d$ we decomposite $S_{z_i}$ with respect to $P_\mathcal{N}=P_{\mathcal{N}\ominus\mathcal{M}}\oplus P_\mathcal{M}$ as follows,
    $$S_{z_i}=\left(
                \begin{array}{cc}
                  S_i^{(1)} & C_i \\
                  0 & S_i^{(2)} \\
                \end{array}
              \right).
    $$
    By the previous lemma, $S_{z_i}$ is $\binom{N+d-2}{d-1}$-multicyclic, and therefore $S_i^{(1)}$ is at most $\binom{N+d-2}{d-1}$-multicyclic on $P_\mathcal{N}\ominus P_\mathcal{M}$. Since
    \begin{eqnarray*}
        [S_i^{(1)*},S_i^{(1)}]&=&[P_{\mathcal{N}\ominus\mathcal{M}}S_{z_i}^*,S_{z_i}P_{\mathcal{N}\ominus\mathcal{M}}]\\
        &=&P_{\mathcal{N}\ominus\mathcal{M}}[S_{z_i}^*,S_{z_i}]P_{\mathcal{N}\ominus\mathcal{M}}+P_{\mathcal{N}\ominus\mathcal{M}}S_{z_i}P_\mathcal{M}S_{z_i}^*P_{\mathcal{N}\ominus\mathcal{M}}
    \end{eqnarray*}
    being the sum of a positive operator and a trace class operator, a generalization of the Berger-Shaw Theorem by Voiculescu\cite[pp. 155, Theorem 2.7 ]{Con} implies $[S_i^{(1)*},S_i^{(1)}]\in\mathcal{L}^1$ and
    $$\tr[S_i^{(1)*},S_i^{(1)}]\leq\binom{N+d-2}{d-1}.$$
    Since
    $$[S_{z_i}^*,S_{z_i}]=\left(
                              \begin{array}{cc}
                                [S_i^{(1)*},S_i^{(1)}]-C_iC_i^* & S_i^{(1)*}C_i-C_iS_i^{(2)*} \\
                                C_i^*S_i^{(1)*}-S_i^{(2)*}C_i & [S_i^{(2)*},S_i^{(2)}]+C_i^*C_i \\
                              \end{array}
                            \right)\in\mathcal{L}^1,$$
    both $[S_i^{(1)*},S_i^{(1)}]-C_iC_i^*$ and $[S_i^{(2)*},S_i^{(2)}]+C_i^*C_i$ are trace class operators. Therefore $C_i\in\mathcal{L}^2$ and $[S_i^{(2)*},S_i^{(2)*}]\in\mathcal{L}^1$, completing the proof of the theorem.
\end{proof}
\begin{rem}
    The proof of Theorem \ref{thm:primary} is also valid for the homogeneous weighted Bergman quotient module case.
\end{rem}
\begin{cor}
    Suppose $I\subset\poly$ is a homogeneous ideal with $Z(I)=V$ and denote $\mathcal{M}=[I]^\bot$. Then $P_\mathcal{M}^\bot M_{z_i}P_\mathcal{M}\in\mathcal{L}^{(2,\infty)}$ and $[M_{z_i}^*,M_{z_i}]P_\mathcal{M}\in\mathcal{L}^{(2,\infty)}$ for $1\leq i\leq d$.
\end{cor}
\begin{proof}
    Let $\mathcal{N}$ be as in the proof of Theorem \ref{thm:primary}, then $C=P_{\mathcal{N}\ominus\mathcal{M}}M_{z_i}P_\mathcal{M}\in\mathcal{L}^2$. By Theorem \ref{thm:primary} and corollary \ref{cor:upright} we have $[M_{z_i}^*P_\mathcal{M},P_\mathcal{M}M_{z_j}]\in\mathcal{L}^1$ and $P_\mathcal{N}^\bot M_{z_i}P_\mathcal{N}\in\mathcal{L}^{(2,\infty)}$. Then we get
    $$P_\mathcal{M}^\bot M_{z_i}P_\mathcal{M}=P_{\mathcal{N}\ominus\mathcal{M}}M_{z_i}P_\mathcal{M}+P_\mathcal{N}^\bot M_{z_i}P_\mathcal{M}\in\mathcal{L}^{(2,\infty)}$$
    and
    $$P_\mathcal{M}[M_{z_i}^*,M_{z_i}]P_\mathcal{M}=[M_{z_i}^*P_\mathcal{M},P_\mathcal{M}M_{z_i}]+P_\mathcal{M}M_{z_i}^*P_\mathcal{M}^\bot M_{z_i}P_\mathcal{M}\in\mathcal{L}^{(1,\infty)}.$$
    Since $[M_{z_i}^*,M_{z_i}]$ is an projection, we have $[M_{z_i}^*,M_{z_i}]P_\mathcal{M}\in\mathcal{L}^{(2,\infty)}$.
\end{proof}
Similarly we can prove the following theorem.
\begin{thm}\label{thm:singlevariety}
    Let $V_\theta$ be a simple homogenous distinguished variety of $\mathbb{D}^d$, and suppose $\mathcal{M}$ is a submodule of $H^2(\mathbb{D}^d)$ such that $Z(\mathcal{M})=V_\theta$, then the quotient module $\mathcal{M}^\perp$ is $1$-essentially normal. Moveover, $P_\mathcal{M}^\bot M_{z_i}P_\mathcal{M}\in\mathcal{L}^{(2,\infty)}$ and $[M_{z_i}^*,M_{z_i}]P_\mathcal{M}\in\mathcal{L}^{(2,\infty)}$ for $1\leq i\leq d$.
\end{thm}
\section{Essential normality of homogenous quotient modules}
\begin{prop}\label{prop:varietydim}
    If $I\subset\poly$ is a homogeneous ideal satisfying $Z(I)\cap\partial\mathbb{D}^d\subset\mathbb{T}^d$, then $Z(I)$ is of complex dimension $1$.
\end{prop}
\begin{proof}
    Define the map $$\Phi:\mathbb{C}^d\backslash\{0\}\to\mathbb{C}^{d-1}\backslash\{0\},~z\mapsto\left(\frac{z_1}{z_d},\ldots,\frac{z_{d-1}}{z_d}\right).$$
    Then $\Phi$ is holomorphic, and maps $Z(I)$ into $\mathbb{T}^{d-1}$. Suppose homogeneous polynomials $f_1,\ldots,f_n$ generate $I$, and define polynomials $g_i\in\mathbb{C}[z_1,\ldots,z_{d-1}]$ by $g(z_1,\ldots,z_{d-1})=f(z_1,\ldots,z_{d-1},1)$. Obviously for $z\in\mathbb{C}^d\backslash\{0\}$, $f_i(z)=0$ if and only if $g_i(\Phi(z))=f_i(\Phi(z),1)=0$. Conversely if $g_i(0)=0$ for each $i$, then we have $f_i(0,\ldots,0,1)=0$, and $Z(I)$ must contain $\{(0,\ldots,0,z)\in\mathbb{C}^d:z\in\mathbb{C}\}$, contracting to the assumption. By now we see that $\Phi(Z(I))=\bigcap_iZ(g_i)$ is an analytic variety contained in $\mathbb{T}^{d-1}$.

    Suppose $V$ is an irreducible component of $\Phi(Z(I))$, and $n=\dim_\mathbb{C}V$. Then there are $d-n$ polynomials $g_1,\ldots,g_{d-n}$ such that $V=Z(g_1,\ldots,g_{d-n})$. Choose a nonsingular point $z_0\in V$, namely the matrix $\left(\frac{\partial g_i}{\partial z_j}(z_0)\right)_{i,j}$ is of rank $d-n$. If $n>0$, by implicit function theorem, there is an analytic mapping $\varphi:\mathbb{C}^n\to\mathbb{C}^d$ satisfying $(g_1,\ldots,g_{d-n})\circ\varphi(z)=z$ on some neighbourhood $U$ of $0$. Therefore $\varphi$ is an analytic embedding from $U$ to $\Phi(Z(I))\subset\mathbb{T}^{d-1}$, which is impossible. Hence $n$ must be zero, and each component of $Z(I)$ must be of complex dimension $1$.
\end{proof}
\begin{thm}\label{thm:main}
    If $I\subset\poly$ is a homogeneous ideal satisfying $Z(I)\cap\partial\mathbb{D}^d\subset\mathbb{T}^d$, then the quotient module $\mathcal{M}=[I]^\bot$ is $(1,\infty)$-essentially normal.
\end{thm}
\begin{proof}
    By proposition \ref{prop:varietydim}, $\dim_\mathbb{C}Z(I)$=1. Suppose $Z(I)\cap\mathbb{D}^d=\bigcup_{k=1}^nV_k$, then $\dim_\mathbb{C}V_k=1$ and therefore $V_k$ is a simple distinguished variety for each $k$. Denote by $J_k$ the prime ideal with zero variety $V_k$ for $k=1,\ldots,n$, and $m$ the maximal ideal at $\{0\}$. Let $I=\bigcap_{k=0}^nI_k$ be the primary decomposition, with $Z(I_0)=\{0\}$ and $Z(I_k)=V_k$. Since $I_0$ is of finite codimension, we may assume $I=I_1\cap\ldots\cap I_n$.

    There is a positive integer $N$ such that $J_k^N\subset I_k$ for each $k$. When $1\leq j,k\leq n$ and $j\neq k$, we can choose a linear polynomial $g_{k,j}\in J_j$, such that $g_{k,j}(z)\neq0$ whenever $z\in V_k$. Since $g_{k,j}$ is homogeneous and $V_k$ is a simple distinguished variety, $|g_{k,j}(z)|$ is constant on $V_k\cap\partial\mathbb{D}^d$. Denote by $\varphi_k=\Pi_{j\neq k}g_{k,j}^N$, then $\varphi_k\in\bigcap_{j\neq k}J_j$ and $|\varphi_k(z)|$ is a nonzero constant on $V_k\cap\partial\mathbb{D}^d$.

    Denote by $\mathcal{M}_k=[I_k]^\bot,\mathcal{N}_k=[J_k^N]^\bot$, and $S_{\varphi_k}'=P_{\mathcal{N}_k}M_{\varphi_k}\mid_{\mathcal{N}_k}$ be the contraction of $M_{\varphi_k}$ on the quotient module $\mathcal{N}_k$. We want to represent $P_{\mathcal{M}}$ by the $P_{\mathcal{M}_k}$'s, and then induce the essential normality of $\mathcal{M}$ from that of the $\mathcal{M}_k$'s.

    First we prove that $S_{\varphi_k}'$ is a Fredholm operator. Without loss of generality we may assume $V_k=V_{(1,\ldots,1)}$. By corollary \ref{cor:isometry}, there is an isometry $S\in B(\mathcal{N}_k)$ which is unitarily equivalent to a multi-shift, and some operator $K_k\in\mathcal{L}^{(1,\infty)}$ such that $S_{\varphi_k}'=\varphi(S_k,\ldots,S_k)+K_k$. Therefore $\sigma_e(S_{\varphi_k}')=\{r\varphi_k(z):z\in\partial\mathbb{D}\}=\varphi_k(V_k\cap\partial\mathbb{D}^d)$, which do not contain $\{0\}$. This implies that $S_{\varphi_k}'$ is Fredholm, and so is $S_{\varphi_k}'^*$. Since $S_{\varphi_k}'^*$ is homogeneous and $\ker S_{\varphi_k}'^*$ is of finite dimension, there is a constant $\rho_k>0$ and positive integer $N_k$ such that $||S_{\varphi_k}'^*h||\geq\rho_k||h||$ whenever $h\in\mathcal{N}_k$ and $\deg h>N_k$. Let $\rho=\min\{\rho_k:k=1,\ldots,n\}$ and $N'=\max\{N_k:k=1,\ldots,n\}$. Denote by $\mathcal{M}'$ the closed subspace of $\mathcal{M}$ spanned by homogeneous polynomials of degree greater than $N'$.

    Denote by $S_{\varphi_k}=P_\mathcal{M}M_{\varphi_k}\mid_\mathcal{M}$ and $T=\sum_{k=1}^nS_{\varphi_k}S_{\varphi_k}^*$. For homogeneous $f\in\mathcal{M}'$, denote by
    $$L=\{g\in H^2(\mathbb{D}^d):g\mbox{ is homogeneous, }\deg g=\deg{f}\}.$$
    We have
    $$f\in\mathcal{M}'\cap L\subset\mathcal{M}\cap L=\left(\bigcap_{k=1}^nI_k\right)^\bot\cap L=L\ominus\left(\bigcap_{k=1}^nI_k\cap L\right)=\bigvee_{k=1}^n L\ominus(I_k\cap L)=L\cap\bigvee_{k=1}^n[I_k]^\bot.$$
    Then we can write $f=\sum_{k=1}^nf_k$ where each $f_k\in[I_k]^\bot=\mathcal{M}_k$ is homogeneous and of the same degree as $f$. Then
    $$||S_{\varphi_k}^*f_k||=||S_{\varphi_k}'^*f_k||\geq\rho||f_k||.$$
    When $j\neq k$, since $\varphi_k\in\bigcap_{j\neq k}J_j$ we have $M_{\varphi_k}^*P_{\mathcal{M}_j}=0$, and therefore $S_{\varphi_k}^*f_j=0$. From this we obtain $S_{\varphi_k}^*f=S_{\varphi_k}^*f_k\in\mathcal{M}_k$. Then
    \begin{eqnarray*}
        \langle Tf,f\rangle&=&\langle\sum_{k=1}^nS_{\varphi_k}S_{\varphi_k}^*f,f\rangle\\
        &=&\sum_{k=1}^n\langle S_{\varphi_k}^*f,S_{\varphi_k}^*f\rangle\\
        &=&\sum_{k=1}^n\langle S_{\varphi_k}^*f_k,S_{\varphi_k}^*f_k\rangle\\
        &\geq&\sum_{k=1}^n\rho^2||f_k||^2\\
        &\geq&\frac{\rho^2}{n}||f||^2.
    \end{eqnarray*}
    Therefore $T$ is bounded below on $\mathcal{M}'$, and must have closed range. Since $T$ keeps degree of polynomials, $T$ maps $\mathcal{M}'$ into a dense subspace of itself. Hence $\mathcal{M}'\subset T\mathcal{M}$ and $T$ is Fredholm.

    Denote by $T'=\sum_{k=1}^nP_{\mathcal{M}_k}S_{\varphi_k}S_{\varphi_k}^*P_{\mathcal{M}_k}$, then $T'\mathcal{M}\subset\bigvee_{k=1}^n\mathcal{M}_k$. By Theorem \ref{thm:singlevariety} for each $k$ we have
    $$P_{\mathcal{M}_k}S_{\varphi_k}^*-S_{\varphi_k}^*P_{\mathcal{M}_k}=P_{\mathcal{M}_k}M_{\varphi_k}^*(P_\mathcal{M}-P_{\mathcal{M}_k})\in\mathcal{L}^{(2,\infty)},$$
    and therefore
    $$T-T'=\sum_{k=1}^nS_{\varphi_k}P_{\mathcal{M}_k}S_{\varphi_k}^*-\sum_{k=1}^nP_{\mathcal{M}_k}S_{\varphi_k}S_{\varphi_k}^*P_{\mathcal{M}_k}\in\mathcal{L}^{(2,\infty)},$$
    which implies $T'$ is Fredholm since $T$ is. By $\bigvee_{k=1}^n\mathcal{M}_k\supset T'\mathcal{M}$ we conclude that $\bigvee_{k=1}^n\mathcal{M}_k$ is closed. Since $\bigvee_{k=1}^n\mathcal{M}_k$ is dense in $\mathcal{M}$, it is just $\mathcal{M}$. Define an operator $$S:\bigoplus_{k=1}^n\mathcal{M}_k\to\mathcal{M},(h_1,\ldots,h_n)\mapsto\sum_{k=1}^nh_k,$$
    then $S$ is surjective hence invertible on $(\ker S)^\bot$. Moreover,
    $$S^*:\mathcal{M}\to\bigoplus_{k=1}^n\mathcal{M}_k,h\mapsto(P_{\mathcal{M}_1}h,\ldots,P_{\mathcal{M}_n}h)$$
    is bounded below. Choose a constant $c>0$ such that $||h||<c||S^*h||$ whenever $h\in\mathcal{M}$. Then for each $h\in H^2(\mathbb{D}^d)$ we have
    \begin{eqnarray*}
        \langle P_\mathcal{M}h,h\rangle&=&||P_\mathcal{M}h||^2\\
        &<&c^2\langle S^*P_\mathcal{M}h,S^*P_\mathcal{M}h\rangle\\
        &=&c^2\sum_{k=1}^n\langle P_{\mathcal{M}_k}h,P_{\mathcal{M}_k}h\rangle\\
        &=&c^2\langle\sum_{k=1}^nP_{\mathcal{M}_k}h,h\rangle,
    \end{eqnarray*}
    which implies $P_\mathcal{M}<c^2\sum_{k=1}^nP_{\mathcal{M}_k}$. Therefore all the eigenvalues of $\sum_{k=1}^nP_{\mathcal{M}_k}$ are greater than $c^{-2}$. Let $\{e_j:j=1,2,\ldots\}$ be an complete system of normalized eigenvectors of $\sum_{k=1}^nP_{\mathcal{M}_k}$, and $\{\lambda_j:j=1,2,\ldots\}$ the associated eigenvalues, then $$\sum_{k=1}^nP_{\mathcal{M}_k}=\sum_{j=1}^\infty\lambda_je_j\otimes e_j.$$
    Let
    $$A=\sum_{j=1}^\infty\lambda_j^{-1}e_j\otimes e_j$$
    and we have $P_\mathcal{M}=A\sum_{k=1}^nP_{\mathcal{M}_k}=\sum_{k=1}^nP_{\mathcal{M}_k}A$. By Theorem \ref{thm:singlevariety}, $P_{\mathcal{M}_k}^\bot M_{z_i}P_{\mathcal{M}_k}\in\mathcal{L}^{(2,\infty)}$ and $[M_{z_i}^*,M_{z_i}]P_{\mathcal{M}_k}\in\mathcal{L}^{(2,\infty)}$ for $1\leq i\leq d$ and $1\leq k\leq n$. Therefore
    $$P_\mathcal{M}^\bot M_{z_i}P_\mathcal{M}=P_\mathcal{M}^\bot M_{z_i}\sum_{k=1}^nP_{\mathcal{M}_k}A=\sum_{k=1}^nP_\mathcal{M}^\bot P_{\mathcal{M}_k}^\bot M_{z_i}P_{\mathcal{M}_k}A\in\mathcal{L}^{(2,\infty)},$$
    and
    $$[M_{z_i}^*,M_{z_i}]P_\mathcal{M}=\sum_{k=1}^n[M_{z_i}^*,M_{z_i}]P_{\mathcal{M}_k}A\in\mathcal{L}^{(2,\infty)}.$$
    Finally
    \begin{eqnarray*}
        [M_{z_i}^*P_\mathcal{M},P_\mathcal{M}M_{z_i}]=P_\mathcal{M}[M_{z_i}^*,M_{z_i}]P_\mathcal{M}-P_\mathcal{M}M_{z_i}^\bot P_\mathcal{M}^\bot M_{z_i}P_\mathcal{M}\in\mathcal{L}^{(1,\infty)},
    \end{eqnarray*}
    completing the proof of the theorem.
\end{proof}

At the end of this section, we give the following proposition, which reveals the structure of semisimple  homogenous quotient module. Following \cite{GWp1}, two subspaces $N_1$ and $N_2$ of a Hilbert space $H$ are said to be asymptotically orthogonal if $P_{N_1}P_{N_2}$ is compact.
\begin{prop}\label{prop:semisimple}
Let $I$ be a homogenous ideal such that $Z(I)\cap \partial \mathbb D^d\subset \mathbb T^d$, $I=I_1\ldots I_n$ be its primary decomposition, with $\dim_{\mathbb{C}}Z(I_i)=1$ for $1\leq i\leq n$. Suppose $\{I_i:1\leq i\leq n\}$ are prime. Then $\{[I_i]^\perp:1\leq i\leq n\}$ are asymptotically orthogonal to each other.
\end{prop}
\begin{proof}
    Assume $V_i=V_{\theta_i}$ for $\theta_i\in\mathbb{T}^d$. Since $[I_i]^\perp$ is prime, we have
    $$[I_i]^\perp=\overline{\span}\{K_{\lambda\theta_i}\in H^2(\mathbb{D}^d):\lambda\in\mathbb{D}\}.$$
    Obviously
    $$K_{\lambda\theta_i}=\sum_{k=1}^\infty\bar{\lambda}^kg_i^k$$
    where $g_i(z)=\langle z,\theta_i\rangle,\forall z\in\mathbb{D}^d$.
    We can see that $\{\frac{g_i^k}{||g_i^k||}:k\in\mathbb{Z}_+\}$ forms an orthonormal basis for $[I_i]^\perp$.
    One can compute for $\lambda\in\mathbb{D}$ that
    \begin{eqnarray*}
        \lambda^k\binom{k+d-1}{d-1}\langle g_i^k,g_j^k\rangle&=&\langle g_i^k,K_{\lambda\theta_j}\rangle\\
        &=&g_i^k(\lambda\theta_j)\\
        &=&\langle\lambda\theta_j,\theta_i\rangle^k\\
        &=&\lambda^k\langle\theta_j,\theta_i\rangle^k,
    \end{eqnarray*}
    which induces $\langle g_i^k,g_j^k\rangle=\binom{k+d-1}{d-1}^{-1}\langle\theta_j,\theta_i\rangle^k$. Since $\langle\theta_i,\theta_i\rangle=d$ we obtain
    $$\lim_{k\to\infty}\frac{|\langle g_i^k,g_j^k\rangle|}{||g_i^k||\cdot||g_j^k||}=\lim_{k\to\infty}\frac{|\langle\theta_j,\theta_i\rangle|^k}{d^k}=0$$
    whenever $|\langle\theta_j,\theta_i\rangle|<d$. This shows that $[I_i]^\perp$ and $[I_j]^\perp$ are asymptotic orthogonal whenever $i\neq j$.
\end{proof}

\begin{rem}\label{rem:ortho}
In the proof of Proposition \ref{prop:semisimple}, it can be seen that if $V_{\theta_1}$ is orthogonal to $V_{\theta_2}$, then  $[J_{\theta_1}]^\perp\ominus \mathbb C$ is orthogonal to $[J_{\theta_2}]^\perp\ominus \mathbb C$, where $\mathbb C$ denotes the 1-dimensional subspace of constant functions.
\end{rem}

\section{Boundary representation of quotient modules}

In this section, we consider the boundary representation of quotient modules. Let us begin by recording Arveson's boundary representation theorem given in \cite{Arv1}.
\begin{thm}\label{thm:ArvBound} Let $\mathcal{A}$ be an irreducible set of operators on a Hilbert space $H$, such that $\mathcal{A}$ contains the identity, and $C^*(\mathcal{A})$ the $C^*$-algebra generated by $\mathcal{A}$ contains the algebra $K(H)$ of compact operators on $H$. Then the identity representation of $C^*(\mathcal{A})$ is a boundary representation for $\mathcal{A}$ if and only if the quotient map $Q: B(H)\to B(H)/K(H)$ is not completely isometric on the linear span of $\mathcal{A}\cup\mathcal{A}^*$.
\end{thm}

To continue, we need to study the essential joint spectrum of $(S_1,\ldots, S_d)$.
\begin{lem}\label{lem:ess}
Let $V$ be a distinguished homogenous variety and $I$ is an ideal such that $Z(I)=V$, then for the quotient module $[I]^\perp$, $V\cap\mathbb T^d\subseteq\sigma_e(S_1,\ldots, S_d)$ and $\sigma_e(S_1,\ldots, S_d)\subseteq V\cap \overline{\mathbb D}^d$.
\end{lem}

\begin{proof} By Theorem \ref{thm:main}, $\lambda_i-S_{z_i}(i=1,\ldots,d)$ are essentially normal. By \cite[Corollary 3.9]{Cur}, the tuple $(\lambda_1-S_{z_1},\ldots,\lambda_d-S_{z_d})$ is Fredholm if and only if $\sum\limits_{i=1}^{d}(\lambda_i-S_{z_i})(\lambda_i-S_{z_i})^*$ is Fredholm. We claim that $\partial V\subseteq \sigma_e(S_1,\ldots, S_d)$. Otherwise, there exists $\underline{\lambda}=(\lambda_1,\ldots,\lambda_d)\in\partial V$ such that $T=\sum\limits_{i=1}^d (\lambda_i-S_{z_i})(\lambda_i-S_{z_i})^*$ is Fredholm. Since $T$ is positive, there is an invertible positive operator $B$ and a compact operator $K$ such that $T=B+K$. Now, take a sequence $\{\underline{\mu}_n\}\subset V\cap \mathbb D^d$ such that $\underline{\mu}_n\to \underline{\lambda}$ as $n\to\infty$. Notice that $\{k_{\underline{\mu}_n}\}$ converges to $0$ weakly, there is a positive number $c$ such that
\begin{eqnarray*}
\lim\limits_{n\to\infty}\langle T k_{\underline{\mu}_n}, k_{\underline{\mu}_n}\rangle=\lim\limits_{n\to\infty}\langle(B+K)k_{\underline{\mu}_n}, k_{\underline{\mu}_n}\rangle=\lim\limits_{n\to\infty}\langle Bk_{\underline{\mu}_n}, k_{\underline{\mu}_n}\rangle\geq c.
\end{eqnarray*}
However, since $\underline{\mu}_n\in V$, $k_{\underline {\mu}_n}\in [I]^\perp$ it holds that
\begin{eqnarray}
\lim\limits_{n\to\infty}\langle T k_{\underline {\mu}_n},k_{\underline {\mu}_n}\rangle=\lim\limits_{n\to\infty} |\underline{\lambda}-\underline{\mu}_n|^2=0,
\end{eqnarray}
contradicting to the former inequality. Hence the claim is proved.

Moreover, for $f\in I$, if $f(S_1,\ldots, S_d)=0$ then the Spectral Mapping Theorem ensures that $\sigma_e(S_1,\ldots, S_d)\subseteq Z(f)$. Hence $\sigma_e(S_1,\ldots, S_d)\subseteq V$. On the other hand, since $\|S_i\|\leq 1$ for each $i$, we have $\sigma_e(S_1,\ldots, S_d)\subseteq V\cap \overline{\mathbb D}^d$.
\end{proof}
The following lemma is easy to verify, and we omit its proof.
\begin{lem}\label{lem:essspec}
Let $X$ be a Banach space, and $T$ is a bounded operator on $X$ with two invariant subspaces $M_1$ and $M_2$ satisfying $X=M_1\dot{+}M_2$, where $\dot{+}$ denotes topological direct sum. Then $\sigma_e(T)=\sigma_e(T_1)\cup \sigma_e(T_2)$.
\end{lem}
Let $V_{\theta_i}(i=1,2)$ be the variety defined by (\ref{Def:V_theta}), and  $I_i(i=1,2)$ be a homogenous ideal with varitey $V_{\theta_i}$. Consider the quotient module $\cal N=[I_1 I_2]^\perp$. Notice that $[I_1]^\perp\cap [I_2]^\perp$ is of finite dimension, and $[I_1]^\perp+[I_2]^\perp$ is closed.  Now, let $\cal N_i=[I_i]^\perp\ominus ([I_1]^\perp\cap [I_2]^\perp)$. Then $$\cal N=(\cal N_1\dot{+} \cal N_2)\oplus ([I_1]^\perp\cap [I_2]^\perp)$$
By Lemma \ref{lem:essspec} we have $\sigma_e(M_{z_i}^*|_N)=\sigma_e(M_{z_i}^*|_{[I_1]^\perp})\cup \sigma_e(M_{z_i}|_{[I_2]^\perp})$. Therefore, to study $\sigma_e(M_{z_i}^*|_N)$ it suffices to study $\sigma_e(M_{z_i}^*|_{[I_i]^\perp})$.
\begin{lem}\label{lem:ess1}
Let $V_\theta$ be defined by (\ref{Def:V_theta}), and $I$ is a homogenous ideal such that $Z(I)=V_\theta$. Then $\sigma_{e}(M_{z_i}|_{[I]^\perp}^*)\subset \mathbb T$.
\end{lem}
\begin{proof}
Let $J_\theta=\sqrt{I}$ and $N$ is a positive integer such that $J_\theta^{N}\subset I\subset J_\theta$. By remark \ref{rem:shif}, there is  an isometry $S$ and a compact operator $K$ such that $M_{z_i}^*|_{[J_\theta^{N}]^\perp}=S^*+K$, and hence $$\sigma_e(M_{z_i}^*|_{[J_\theta^{N}]^\perp})=\sigma_e(S^*)=\mathbb T.$$ Because $[I]^\perp$ is an invariant subspace of $M_{z_i}^*|_{[J_\theta^{N}]^\perp}$, we have $\sigma_{e}(M_{z_i}|_{[I]^\perp}^*)\subset \mathbb T$.
\end{proof}
By Lemma \ref{lem:essspec}, for distinguished variety $V$ and homogenous ideal $I$ with $Z(I)=V$, we have \begin{eqnarray}\label{for:ess}\sigma_e(M_{z_i}^*|_{[I]^\perp})\subset \mathbb T.\end{eqnarray}
\begin{prop}
Let $I$ be a homogenous ideal with distinguished variety, then $$\sigma_e(S_{z_1},\ldots,S_{z_d})=V\cap \mathbb T^d.$$
\end{prop}
\begin{proof}
By Lemma \ref{lem:ess}, it suffices to show $\sigma_e(S_{z_1},\ldots,S_{z_d})\subset\partial \mathbb D^d$. By (\ref{for:ess}), $\lambda-S_{z_i}$ is Fredholm if $|\lambda|<1$. It follows that for $\underline{\lambda}=(\lambda_1,\ldots,\lambda_d)\in\mathbb D^d$, $\lambda_i-S_{z_i}$ are Fredholm. Therefore $\sigma_e(S_{z_1},\ldots,S_{z_d})\subset \mathbb T^d$ by \cite{Cur}, which completes the proof.
\end{proof}

Suppose that $V$ is a distinguished homogenous variety, with the decomposition $V=V_1\cup\ldots\cup V_n$, such that each $V_i$ is irreducible with $\dim V_i=1$. Let $I$ be an ideal such that $Z(I)=V$, and we have the associated primary decomposition $I=I_0\cap I_1\cap\ldots\cap I_n$, such that $Z(I_0)=\{0\}$ and $Z(I_i)=V_i$ for $i=1,\ldots,n$. Since $\sqrt{I_0\cap I_1}=\sqrt{I_0}\cap\sqrt{I_1}=\sqrt{I_1}$, we can assume $I=I_1\cap\ldots\cap I_n$ without loss of generality.

\begin{thm}\label{thm:boundary}
With the above notations, if some $I_i$ is not prime, then the identity representation of  $C^*([I]^\perp)$ is  a boundary representation for $\cal B(S_1,\ldots, S_d)$.
\end{thm}

\begin{proof}
It is easy to see that $\{S_1,\ldots,S_d\}$ is a irreducible set and $C^*([I]^\perp)$ contains all the compact operators on $[I]^\perp$.  Without loss of generality, suppose $I_1$ is not prime. Then by the proof of Theorem \ref{thm:main}, we can find a polynomial $\varphi\in I_2\cap\ldots\cap I_n$ such that $\varphi\notin\sqrt{J_1}$. Choose $g\in\sqrt{I_1}-I_1$ and let $f=\varphi g$, then we have $f\in \sqrt{I_1}\cap I_2\cap\ldots\cap I_n$ and $f\not\in {I_1}\cap I_2\cap\ldots\cap I_n$. Then $f|_{V}=0$ and $S_f=P_{[I]^\perp}M_f|_{[I]^\perp}\not=0$. By the Spectral Mapping Theorem,
$$
\sigma_e(S_f)=f(\sigma_e(S_1,\ldots,S_d))=\{0\}.
$$
Since $S_f$ is essentially normal, it must be compact. On the other hand, since $f\not\in {I_1}\cap I_2\cap\ldots\cap I_n$, we have $S_f\not=0$ and then $\|S_f\|>\|S_f\|_e=0$. By Arveson's boundary representation theorem, the lemma is proved.
\end{proof}
To study the negative proposition of Theorem \ref{thm:boundary}, we have the following result. Its proof follows from \cite[page 292, Corollary 2]{Arv1} and the proof of Proposition \ref{prop:shift}.
\begin{prop}
Suppose that $I$ is a prime homogenous ideal such that $Z(I)$ is a distinguished variety. Then the identity representation of $C^*([I]^\perp)$ is not a boundary representation for $\cal B(S_1,\ldots, S_d)$.
\end{prop}
\begin{exam}
With the above notations, suppose that all the $I_i$'s are prime. If the $Z(I_i)$'s are orthogonal to each other, then the $[I_i]^\perp\ominus \mathbb C$'s are orthogonal to each other by Remark \ref{rem:ortho}. It follows that in this case, the identity representation of $C^*([I]^\perp)$ is not a boundary representation for $\cal B(S_1,\ldots, S_d)$.
\end{exam}

\vskip2.5mm\noindent {\bf Acknowledgements.} Both the authors would like to thank Professor Kunyu Guo for his encouragements and interests.


\begin{thebibliography}{99}




\bibitem{AM} J. Agler and J. E. McCarthy, \emph{Distinguished varieties}, Acta Math. 194 (2005), no. 2, 133-153, DOI: 10.1007/BF02393219

\bibitem{Arv} W. Arveson, \emph{Subalgebras of $C^*$-algebras}. Acta Math. 123 (1969) 141-224.

\bibitem{Arv1} W. Arveson, \emph{Subalgebras of $C^*$-algebras. II.} Acta Math. 128 (1972), no. 3-4, 271-308.

\bibitem{Ar1} W. Arveson, \emph{The Dirac operator of a commuting d-tuple,} J. Funct. Anal. 189 (2002), 53-79.

\bibitem{Ar3} W. Arveson, \emph{Quotients of standard Hilbert modules,} Trans. Amer. Math. Soc. 359 (2007), no. 12, 6027-6055.

\bibitem{BDF} L. Brown, R. Douglas and P. Fillmore, \emph{Extension of $C^*$-algebras and K-homology}, Ann. of Math. 105 (1977), 265-324.

\bibitem{CG} X. Chen and K. Guo, \emph{Analytic Hilbert Modules,} CRC Research Notes, 413, 2003.

\bibitem{Cla} D. N. Clark, \emph{Restrictions of $H^p$ functions in the polydisk}, Amer. J. Math. 110 (1988), 1119-1152.

\bibitem{Con} J. B. Conway, \emph{The theory of subnormal operators}, American Mathematical Soc., 1991.

\bibitem{Cur} R. Curto, \emph{Fredholm and invertible n-tuples of operators. The deformation problem.} Trans. Amer. Math. Soc. 266 (1981), no. 1, 129-159.

\bibitem{Dou1} R. G. Douglas, \emph{Essentially reductive Hilbert modules}, J. Oper. Theory. 55 (2006), 117-133.

\bibitem{Dou2} R. G. Douglas, \emph{Essentially reductive Hilbert modules. II}. Hot topics in operator theory, 79-87, Theta Ser. Adv. Math., 9, Theta, Bucharest, 2008.

\bibitem{DP} R. G. Douglas and V. I. Paulsen, \emph{Hilbert Modules over Function Algebras,} Longman Research Notes, 217, 1989.

\bibitem{DM} R. G. Douglas and G. Misra,\emph{Some calculations for Hilbert modules}, J.Orissa Math. Soc. 12-15 (1993-1996), 75-85.

\bibitem{DTY} R. G. Douglas, X. Tang and G. Yu, \emph{An Analytic Grothendieck Riemann Roch Theorem},  arXiv: 1404.4396.

\bibitem{DW} R. G. Douglas and K. Wang, \emph{A harmonic analysis approach to essential normality of principal submodules}, J. Funct. Anal. 261 (2011), 3155-3180.

\bibitem{Eng} M. Engli\v{s}, \emph{Density of algebras generated by Topelitz operators on Bergman spaces}, Ark. Mat. 30 (1992), 227-240.

\bibitem{Fel} S. Feldman, \emph{The Berger-Shaw theorem for cyclic subnormal operators}, Indiana University Mathematics Journal, 46 (1997), no. 3, 741-752.

\bibitem{FR1} S. Ferguson and R. Rochberg, \emph{Higher order Hilbert-Schmidt Hankel forms and tensors of analytical kernels}, J. Mathematica Scandinavica, 96 (2005), no. 1, 117-146.

\bibitem{FR2} S. Ferguson and R. Rochberg, \emph{Description of certain quotient Hilbert modules}, J. Oper. Theory, 20 (2006), 93-109.


\bibitem{Guo} K. Guo, \emph{Defect operators for submodules of $H^2_d$}, J. Reine Angew. Math. 573 (2004), 181-209.

\bibitem{GWk1} K. Guo and K. Wang, \emph{Essentially normal Hilbert modules and K-homology}, Math. Ann., 340 (2008), 907-934.

\bibitem{GWk2} K. Guo and K. Wang, \emph{Essentially normal Hilbert modules and K-homology. II: Quasihomogeneous Hilbert modules over two dimensional unit ball}, J. Ramanujan Math. Soc. 22 (2007), 259-281.

\bibitem{GWk3} K. Guo and K. Wang, \emph{Beurling type quotient modules over the bidisk and boundary representations}, J. Funct. Anal. 257 (2009), 3218-3238.

\bibitem{GWp1} K. Guo and P. Wang, \emph{Essentially normal Hilbert modules and K-homology, III. Homogenous quotient modules of Hardy modules on the bidisk,} Sci. China. Ser. A 50 (2007), no. 3, 387-411.

\bibitem{GWp2} K. Guo and P. Wang, \emph{Essentially normal Hilbert modules and K-homology IV: Quasi-homogenous quotient modules of Hardy module on the polydisks,} Sci. China Math, 55 (2012), no. 8, 1613-1626.

\bibitem{GWZ} K. Guo, K. Wang and G. Zhang, \emph{Trace formulas and $p$-essentially normal properties of quotient modules on the bidisk}, J. Oper. Theory, 67 (2012), no. 2, 511-535.

\bibitem{Har} R. Hartshorne, \emph{Algebraic Geometry}, Springer 52, 1977.

\bibitem{KS} M. Kennedy and O. Shalit \emph{Essential normality, Essential norms and hyperrigidity}, J. Funct. Anal., 268 (2015), 2990-3016.


\bibitem{Ru} W. Rudin, \emph{Function theory in polydiscs}, W.A.Benjamin, INC, 1969.

\bibitem{Sha} O. Shalit, \emph{Stable polynomial division and essential normality of graded Hilbert modules}, J. London Math. Soc., 83 (2011), 273-289.

\bibitem{Wa} P. Wang, \emph{The essential normality of $N_\eta$-type quotient module of Hardy module on the polydisc}, Proc. Amer. Math. Soc. 142 (2014), no. 1, 151-156.

\end{thebibliography}
\end{document}